\documentclass[reqno]{amsart}
\usepackage{hyperref}
\usepackage{amssymb}

\allowdisplaybreaks

\theoremstyle{plain}
\newtheorem{theorem}{Theorem}[section]

\theoremstyle{definition}

\theoremstyle{remark}
\newtheorem{remark}[theorem]{Remark}
\numberwithin{equation}{section}

\begin{document}
\title[Semilinear heat equations with nonlinear boundary conditions]
{Influence of variable coefficients on global existence of solutions of semilinear heat equations with nonlinear boundary conditions}

\author[A. Gladkov]{Alexander Gladkov}
\address{Alexander Gladkov \\ Department of Mechanics and Mathematics
\\ Belarusian State University \\  4  Nezavisimosti Avenue \\ 220030
Minsk, Belarus  and  Peoples' Friendship University of Russia (RUDN University) \\  6 Miklukho-Maklaya street \\  117198 Moscow,  Russian Federation}    \email{gladkoval@mail.ru }

\author[M. Guedda]{Mohammed~Guedda}
\address{Mohammed~Guedda \\ Universit\'{e} de Picardie,
LAMFA, CNRS, UMR 7352, 33 rue Saint-Leu, F-80039,  Amiens, France}
\email{mohamed.guedda@u-picardie.fr}


\begin{abstract}
We consider semilinear parabolic equations with nonlinear boundary conditions.
We give conditions which guarantee global existence of solutions as well as blow-up
in finite time of all solutions with nontrivial initial data. The results depend on the behavior
of variable coefficients as $t \to \infty.$
\end{abstract}

\maketitle


\section{Introduction}

 We investigate the global solvability and blow-up in finite time for
semilinear heat equation
\begin{equation}\label{1e}
   u_{t} = \Delta u + \alpha (t)  f(u)  \,\,\, \textrm{for} \,\,\,
x \in \Omega, \,\,\, t>0,
\end{equation}
with nonlinear boundary condition
\begin{equation}\label{1b}
\frac{\partial u(x,t)}{\partial\nu} = \beta (t)  g(u)  \,\,\, \textrm{for} \,\,\, x \in\partial\Omega, \,\,\,  t > 0,
\end{equation}
and initial datum
\begin{equation}\label{1i}
   u(x,0)= u_0(x)  \,\,\, \textrm{for} \,\,\,  x  \in \Omega,
\end{equation}
where $\Omega$ is a bounded domain in $\mathbb{R}^n$ for $n \geq
1$ with smooth boundary $\partial  \Omega,$ $\nu$ is the unit
exterior normal vector on the boundary $\partial\Omega.$
Here $f (u)$ and $g (u)$ are nonnegative continuous functions for $ u \geq  0,\,$
$\alpha (t)$  and $\beta (t)$ are nonnegative continuous functions for $ t \geq  0,\,$
$ u_0(x)\in C^1(\overline{\Omega}),\,$ $ u_0(x)\geq 0$ in $\overline\Omega$
and satisfies boundary condition (\ref{1b}) as $t=0.$ We will consider nonnegative
classical solutions of (\ref{1e})--(\ref{1i}).

Blow-up problem for parabolic equations with reaction term in general form
 were considered in many papers (see, for example,  \cite{FM} -- \cite{Y}  and the references therein).
  For the global existence and blow-up of solutions for linear parabolic equations
  with $\beta (t) \equiv 1$ in (\ref{1b}), we refer to previous studies \cite{LP} -- \cite{RR}.
  In particular, Walter \cite{W} proved that if $g (s)$ and $g' (s)$
  are continuous, positive and increasing for large $s,$
  a necessary and sufficient condition for global existence is
$$
\int^{+\infty}  \frac{ds}{g(s)g'(s)} = + \infty.
$$

Some papers are devoted to blow-up phenomena in parabolic problems with time-dependent coefficients
(see, for example,  \cite{M} -- \cite{ZL}).
So, it follows from results of Payne and Philippin \cite{PP}
blow-up of all nontrivial solutions for (\ref{1e})--(\ref{1i})
with $\beta (t) \equiv 0$ under the conditions (\ref{211}) and
$$
f(s) \geq z(s) > 0, \,\, s > 0,
$$
where $z$ satisfies
\begin{equation*}
\int_a^{+\infty}  \frac{ds}{z(s)} < + \infty \,\,\, \textrm{for any} \,\,\, a > 0
\end{equation*}
and  Jensen's inequality
\begin{equation}\label{Re2}
\frac{1}{|\Omega|}\int_{\Omega}  z(u) \, dx  \geq  z \left( \frac{1}{|\Omega|}\int_{\Omega}  u \, dx \right).
\end{equation}
In (\ref{Re2}), $|\Omega| $ is the volume of $\Omega.$

The aim of our paper is study the influence of variable coefficients $\alpha (t)$ and $\beta (t)$ on the
 global existence and blow-up of classical solutions of (\ref{1e})--(\ref{1i}).

 This paper is organized as follows.
 Finite time blow-up of all nontrivial solutions is proved in Section 2.
In Section 3, we present the global existence of solutions for small initial data.


\section{Finite time blow-up}\label{FT}

In this section, we give conditions for blow-up in finite time of all nontrivial solutions of (\ref{1e}) -- (\ref{1i}).

Before giving our main results, we state a comparison principle which has been proved in \cite{F}, \cite{Walter} for more general problems.
Let $Q_T=\Omega\times(0,T),\;S_T=\partial\Omega\times(0,T),$
$\Gamma_T=S_T\cup\overline\Omega\times\{0\}$, $T>0.$
\begin{theorem}\label{comp-prins}
     Let $v(x,t), w(x,t) \in C^{2,1}(Q_T)\cap C^{1,0}(Q_T\cup\Gamma_T)$  satisfy the inequalities:
\begin{equation*}\label{}
   v_{t} - \Delta v - \alpha (t)  f(v) <  w_{t} - \Delta w - \alpha (t)  f(w)   \,\,\, \textrm{in} \,\,\, Q_T,
\end{equation*}
\begin{equation*}\label{c1}
\frac{\partial v(x,t)}{\partial\nu} - \beta (t)  g(v) <  \frac{\partial w(x,t)}{\partial\nu} - \beta (t)  g(w)   \,\,\, \textrm{on} \,\,\, S_T,
\end{equation*}
\begin{equation*}\label{}
   v(x,0) <  w(x,0)  \,\,\, \textrm{in} \,\,\, \overline{\Omega}.
\end{equation*}
Then
\begin{equation*}\label{}
   v(x,t) <  w(x,t)  \,\,\, \textrm{in} \,\,\, Q_T.
\end{equation*}
\end{theorem}
The first our blow-up result is the following.
\begin{theorem}\label{Th1}
Let $g(s)$ be a nondecreasing positive function for $s >0$ such that
\begin{equation}\label{01}
\int^{+\infty}  \frac{ds}{g(s)} < + \infty
\end{equation}
and
\begin{equation}\label{1}
\int_0^{+\infty}  \beta (t)  \,dt = +\infty.
\end{equation}
Then any nontrivial nonnegative solution of (\ref{1e}) -- (\ref{1i}) blows up in finite time.
\end{theorem}
\begin{proof}
We suppose that $u(x,t)$ is a nontrivial nonnegative solution which exists in $Q_T$ for any positive $T.$
Then for some $T>0$ there exists $(\overline x, \overline t) \in Q_T$ such that $u(\overline x, \overline t) > 0.$
Since $u_t - \Delta u = \alpha (t) f(u) \geq 0,$
by strong maximum principle $u(x,t) > 0$ in $Q_T \setminus \overline {Q_{\overline t}}.$
Let $u(x_\star,t_\star)=0$ in some point $(x_\star,t_\star) \in S_T \setminus \overline {S_{\overline t}}.$
According to Theorem~3.6 of~\cite{Hu} it yields $\partial u(x_\star,t_\star)/ \partial \nu < 0,$
which contradicts the boundary condition (\ref{1b}). Thus, $u(x,t) > 0$ in $Q_T \cup S_T \setminus \overline {Q_{\overline t}}.$
Then there exists $t_0 > \overline t$ such that $\beta (t_0) > 0$ and
\begin{equation}\label{t0}
 \min_{\overline\Omega} u(x,t_0) > 2 \sigma,
\end{equation}
where $\sigma$ is a positive constant.

Let $G_N(x,y;t-\tau)$ denote the Green's function for the heat equation
given by
$$
u_t - \Delta u =0 \,\,\, \textrm{for}\,\,\, x \in \Omega, \, t>0
$$
with homogeneous Neumann boundary condition.
We note that the Green's function has the following properties (see, for
example, \cite{Kahane}, \cite{HY}):
    \begin{equation}\label{G1}
G_N(x,y;t-\tau) \geq 0, \; x,y \in \Omega, \; 0 \leq \tau < t,
    \end{equation}
    \begin{equation}\label{G2}
\int_{\Omega}{G_N(x,y;t-\tau)} \, dy=1, \; x\in\Omega, \; 0 \leq \tau < t,
    \end{equation}
     \begin{equation}\label{G3}
        G_N(x,y;t-\tau)\geq c_1, \;  x,\,y\in\overline\Omega,\; t-\tau \geq   \varepsilon,
    \end{equation}
 \begin{equation*}\label{G4}
        |G_N(x,y;t-\tau)-1/|\Omega||\leq c_2 \exp[-c_3 (t-\tau)],\; x,\,y\in\overline\Omega,\; t-\tau \geq \varepsilon,
    \end{equation*}
\begin{equation*}\label{G5}
\int_{\partial\Omega}G_N(x,y;t-\tau)\,dS_y \leq
\frac{c_4}{\sqrt{t-\tau}}, \; x\in\overline\Omega, \; 0< t-\tau \leq
\varepsilon,
\end{equation*}
for some small $\varepsilon>0.$ Here by $c_i\,(i\in \mathbb N)$ we denote
positive constants.

Now we introduce several auxiliary functions.
We suppose that $h(s) \in C^1((0, +\infty))\cap C([0, +\infty)),$ $h(s) >0$ for
$s >0,\,$  $h'(s) \geq 0$ for $s >0,\,$  $g(s) \geq h(s)$ and
\begin{equation*}\label{121}
\int^{+\infty}  \frac{ds}{h(s)} < + \infty.
\end{equation*}
Let $\xi (t)$ be a positive continuous function for $t \geq t_0$ such that
\begin{equation}\label{122}
\int^{+\infty}_{t_0}  \xi (t) \, dt  < \frac{ \sigma}{2}
\end{equation}
and $\gamma (t)$ be a positive continuous function for $t \geq t_0$ such that
$\gamma (t_0) = \beta (t_0) h(2\sigma)$ and
\begin{equation}\label{123}
\int^{t}_{t_0}  \gamma (\tau) \int_{\partial\Omega}G_N(x,y;t-\tau)\,dS_y \, d\tau  < \frac{ \sigma}{2}
\,\,\, \textrm{for}\,\,\,  x \in \overline\Omega, \; t \geq t_0.
\end{equation}

We consider the following problem
\begin{equation}\label{1123}
\left\{
  \begin{array}{ll}
     v_{t} = \Delta v - \xi (t)  \,\,\, \textrm{for} \,\,\,
x \in \Omega, \,\,\, t>t_0, \\
\frac{\partial v(x,t)}{\partial\nu} = \beta (t) h(v) - \gamma (t)
\,\,\, \textrm{for} \,\,\, x \in\partial\Omega, \,\,\,  t > t_0, \\
   v (x,t_0)= 2 \sigma  \,\,\, \textrm{for} \,\,\,  x  \in \Omega.
  \end{array}
\right.
\end{equation}
To find lower bound for $v(x,t)$ we represent (\ref{1123}) in equivalent form
    \begin{eqnarray}\label{le:equat}
v (x,t)&=& 2 \sigma  \int_\Omega G_N(x,y;t) \, dy -
\int_{t_0}^t  \int_\Omega  G_N(x,y;t-\tau)  \xi(\tau) \,dy \,d\tau \nonumber\\
&&+\int_{t_0}^t \int_{\partial\Omega} G_N(x,y;t-\tau) \left( \beta (\tau) h(v) - \gamma (\tau) \right)  \,dS_y \, d\tau.
    \end{eqnarray}
 Using (\ref{122}), (\ref{123}) and the properties of the Green's function (\ref{G1}), (\ref{G2}), we obtain from (\ref{le:equat})
\begin{equation}\label{v}
v (x,t) \geq  2 \sigma  - \int_{t_0}^t   \xi(\tau)  \, d\tau
- \int_{t_0}^t \gamma (\tau) \int_{\partial\Omega} G_N(x,y;t-\tau)  dS_y \, d\tau > \sigma.
\end{equation}

As in~\cite{RR} we put
\begin{equation*}\label{}
m(t) = \int_\Omega \int_{v(x,t)}^{+\infty}  \frac{ds}{h(s)} \, dx.
\end{equation*}
We observe that $m(t)$ is well defined and positive for $t \geq  t_0.$
Since $v(x,t)$ is the solution of (\ref{1123}), we get
 \begin{eqnarray*}
m'(t) &=& - \int_\Omega  \frac{v_t}{h(v)} \, dx = - \int_\Omega  \frac{\Delta v}{h(v)} \, dx
+ \xi (t) \int_\Omega  \frac{dx}{h(v)} \\
&=& - \int_\Omega div \left( \frac{\nabla v}{h(v)} \right) \, dx - \int_\Omega  \frac{h'(v) ||\nabla v||^2}{h^2(v)} \, dx
+ \xi (t) \int_\Omega  \frac{dx}{h(v)}.
    \end{eqnarray*}

Applying the inequality $h'(v) \geq 0,$ Gauss theorem, the boundary condition in (\ref{1123})
and (\ref{v}), we obtain for $t \geq  t_0$
\begin{equation}\label{m}
m'(t) \leq - \int_{\partial \Omega } \frac{1}{h(v)}  \frac{\partial v}{\partial \nu} \, dS
+ \xi (t) \frac{|\Omega|}{h(\sigma)}
 \leq - |\partial \Omega| \beta (t) +  \frac{|\Omega|\xi (t) + |\partial \Omega|\gamma (t)}{h(\sigma)}.
\end{equation}
Due to (\ref{1}), (\ref{G3}) -- (\ref{123})  $m(t)$ is negative for large values of $t.$
Hence $v(x,t)$ blows up in finite time $T_0.$ Applying Theorem~\ref{comp-prins} to $v(x,t)$  and $u(x,t)$
in $Q_T \setminus \overline{Q_{t_0}}$  for any $T \in  (t_0, T_0),$ we prove the theorem.
\end{proof}


\begin{remark}
If $u_0 (x)$ is positive in $\overline{\Omega}$ we can obtain an upper bound for blow-up time of the solution.
We put $t_0 =0$ and $v(x, 0) = u_0 (x) - \varepsilon$ in (\ref{1123}) for $\varepsilon \in (0, \min_{\overline{\Omega}} u_0 (x)).$
Integrating (\ref{m}) over $[0, T],$ we have
\begin{equation*}
m(t) \leq m(0) - |\partial \Omega| \int_{0}^{T} \beta (t) \, dt
+ \int_{0}^{T}  \frac{|\Omega|\xi (t) + |\partial \Omega|\gamma (t)}{h(\sigma)} \, dt.
\end{equation*}
Since $m(t)>0$ and $\varepsilon, \, \xi (t), \, \gamma (t)$ are arbitrary  we conclude that the solution of
(\ref{1e})--(\ref{1i}) blows up in finite time $T_b,$ where $T_b \leq T$ and
$$
\int_\Omega \int_{u_0(x)}^{+\infty}  \frac{ds}{h(s)} \, dx = |\partial \Omega| \int_{0}^{T} \beta (t) \, dt.
$$
 \end{remark}


\begin{remark}
We note that (\ref{1e})--(\ref{1i}) with $u_0 (x) \equiv 0$
may have trivial and blow-up solutions under the assumptions of Theorem~\ref{Th1}.
Indeed, let the conditions of Theorem~\ref{Th1} hold,
$\alpha (t) \equiv 0, \,$ $\beta (t) \equiv 1$
and $g(u) = u^p, $  $u \in [0, \gamma]$ for some $\gamma >0$ and $0<p<1.$
As it was proved in \cite{CR}, problem (\ref{1e})--(\ref{1i}) has trivial and positive for $t >0$
solutions and last one blows up in finite time by Theorem~\ref{Th1}.
 \end{remark}


To prove next blow-up result for (\ref{1e}) -- (\ref{1i}) we need
a comparison principle with unstrict inequality in the boundary condition.
\begin{theorem}\label{comp2}
     Let $\delta > 0$ and $v(x,t), w(x,t) \in C^{2,1}(Q_T)\cap C^{1,0}(Q_T\cup\Gamma_T)$  satisfy the inequalities:
\begin{equation*}\label{}
   v_{t} - \Delta v - \alpha (t)  f(v) + \delta <  w_{t} - \Delta w - \alpha (t)  f(w)   \,\,\, \textrm{in} \,\,\, Q_T,
\end{equation*}
\begin{equation*}\label{}
\frac{\partial v(x,t)}{\partial\nu}  \leq  \frac{\partial w(x,t)}{\partial\nu}   \,\,\, \textrm{on} \,\,\, S_T,
\end{equation*}
\begin{equation*}\label{}
   v(x,0) <  w(x,0)  \,\,\, \textrm{in} \,\,\, \overline{\Omega}.
\end{equation*}
Then
\begin{equation*}\label{}
   v(x,t) \leq  w(x,t)  \,\,\, \textrm{in} \,\,\, Q_T.
\end{equation*}
\end{theorem}
\begin{proof}
Let $\tau$  be any positive constant such that $\tau < T$ and a positive function $\gamma (x) \in C^2 (\overline\Omega)$ satisfy the following inequality
$$
\frac{\partial\gamma (x)}{\partial\nu} > 0  \,\,\, \textrm{on} \,\,\, \partial\Omega.
$$
For positive $\varepsilon$ we introduce
\begin{equation}\label{w}
  w_\varepsilon (x,t) =  w(x,t) + \varepsilon \gamma (x).
\end{equation}
Obviously,
\begin{equation*}\label{}
   v(x,0) <  w_\varepsilon (x,0)  \,\,\, \textrm{in} \,\,\, \overline{\Omega}, \quad
\frac{\partial v(x,t)}{\partial\nu}  <  \frac{\partial w_\varepsilon (x,t)}{\partial\nu}   \,\,\, \textrm{on} \,\,\, S_\tau.
\end{equation*}
Moreover,
\begin{equation*}\label{}
   v_{t} - \Delta v - \alpha (t)  f(v) <  w_{\varepsilon t} - \Delta w_\varepsilon  - \alpha (t)  f(w_\varepsilon)   \,\,\, \textrm{in} \,\,\, Q_\tau,
\end{equation*}
if we take $\varepsilon$ so small that
$$
\delta > \varepsilon \Delta \gamma + \alpha (t)  [ f(w + \varepsilon \gamma)  - f(w) ]  \,\,\, \textrm{in} \,\,\, Q_\tau.
$$
Applying Theorem~\ref{comp-prins} with $\beta (t) \equiv 0,$ we obtain
\begin{equation*}\label{}
   v(x,t) <  w_\varepsilon (x,t)  \,\,\, \textrm{in} \,\,\, Q_\tau.
\end{equation*}
Passing to the limits as $\varepsilon \to 0$ and $\tau \to T,$ we prove the theorem.
\end{proof}

\begin{theorem}\label{Th2}
Let  $f(s) >0$ for $s >0,$
\begin{equation}\label{21}
\int^{+\infty}  \frac{ds}{f(s)} <+ \infty
\end{equation}
and
\begin{equation}\label{211}
\int_0^{+\infty}  \alpha (t)  \,dt = +\infty.
\end{equation}
Then any nontrivial nonnegative solution of (\ref{1e}) -- (\ref{1i}) blows up in finite time.
\end{theorem}
\begin{proof}
We suppose that $u(x,t)$ is a nontrivial nonnegative solution which exists in $Q_T$ for any positive $T.$
In Theorem~\ref{Th1} we proved  (\ref{t0}).
Let $\xi (t)$ be a positive continuous function for $t \geq t_0$ such that
\begin{equation}\label{22}
\max_{[\sigma, 2 \sigma]} f(s) \int^{+\infty}_{t_0}  \xi (t) \, dt  <  \sigma.
\end{equation}
We consider the following auxiliary problem
\begin{equation}\label{23}
\left\{
  \begin{array}{ll}
    v'(t) = \alpha (t)  f(v) - \xi (t)  f(v)  , \; t>t_0, \\
    v(t_0) = 2 \sigma.
  \end{array}
\right.
\end{equation}
We prove at first that
\begin{equation}\label{24}
 v(t) > \sigma  \,\,\, \textrm{for} \,\,\, t \geq t_0.
\end{equation}
Suppose there exist $t_1$ and $t_2$ such that
$$
t_2 > t_1 \geq t_0, \; v(t_1) = 2 \sigma, \; v(t_2) = \sigma,
$$
and
$$
  v(t)  > \sigma \,\,\, \textrm{for} \,\,\, t \in [t_0, t_2)  \,\,\, \textrm{and} \,\,\,
  v(t)  \leq 2 \sigma \,\,\, \textrm{for} \,\,\, t \in [t_1, t_2].
$$
Integrating the equation in (\ref{23}) over $[t_1, t_2],$ we have due to (\ref{22})
$$
  v(t_2)  \geq - \max_{[\sigma, 2 \sigma]} f(s)  \int^{t_2}_{t_1}  \xi (t) \, dt + v(t_1) > \sigma.
$$
A contradiction proves (\ref{24}).

From (\ref{23}) we obtain
\begin{equation}\label{25}
\int^{v(t)}_{2\sigma}  \frac{ds}{f(s)} = \int_{t_0}^{t} [ \alpha (\tau) - \xi (\tau) ] \,d\tau.
\end{equation}
By (\ref{21}) -- (\ref{22}) the left side of (\ref{25}) is finite and the right side of (\ref{25}) tends to infinity as $t \to \infty.$
Hence the solution of (\ref{23}) blows up in finite time $T_0.$ Applying Theorem~\ref{comp2} to $v(t)$  and $u(x,t)$ in $Q_T \setminus \overline{Q_{t_0}}$  for any $T \in  (t_0, T_0),$ we prove the theorem.
\end{proof}


\begin{remark}
If $u_0 (x)$ is positive in $\overline{\Omega}$ we can obtain an upper bound for blow-up time of the solution.
Taking $t_0 =0,$  we conclude from (\ref{25}) that the solution of
(\ref{1e})--(\ref{1i}) blows up in finite time $T_b,$ where $T_b \leq T$ and
$$
\int_{\min_{\overline{\Omega}} u_0 (x)}^{+\infty}  \frac{ds}{f(s)} = \int_{0}^{T}  \alpha (t) \,dt.
$$
 \end{remark}


\begin{remark}\label{Rem11}
Theorem~\ref{Th2} does not hold if $f(s)$ is not positive for $s>0.$
To show this we suppose that $f(u_1) = 0$ for some $u_1 > 0,\,$ $\beta (t) \equiv 0,\,$ $u_0 (x) = u_1.$
Then problem (\ref{1e}) -- (\ref{1i}) has the solution $u (x,t) = u_1.$
\end{remark}


\begin{remark}\label{Rem2}
We note that (\ref{21}) is necessary condition for blow-up of solutions of (\ref{1e})--(\ref{1i}) with $\beta (t) \equiv 0.$
Let $f(s) >0$ for $s >0$ and
\begin{equation*}
\int^{+\infty}  \frac{ds}{f(s)} = + \infty.
\end{equation*}
Then any solution of (\ref{1e})--(\ref{1i}) is global. Indeed, let $u(x,t)$ be a nontrivial
solution of (\ref{1e})--(\ref{1i}). Then there exist $t_0 \geq 0$ and $x \in \Omega$ such that $u(x,t_0) > 0.$

We consider the following problem
\begin{equation}\label{231}
\left\{
  \begin{array}{ll}
    v'(t) = (\alpha (t) + \xi (t))  f(v)  , \; t>t_0, \\
    v(t_0) > \max_{\overline\Omega} u(x,t_0) > 0,
  \end{array}
\right.
\end{equation}
where $\xi (t)$ is some positive continuous function for $t \geq t_0.$
Obviously, $v (t)$ is global solution of (\ref{231}).
Applying Theorem~\ref{comp2} to $u(x,t)$  and $v(t)$
in $Q_T \setminus \overline{Q_{t_0}}$  for any $T > t_0,$ we prove the theorem.
\end{remark}

\begin{remark}
Problem (\ref{1e})--(\ref{1i}) with $u_0 (x) \equiv 0$
may have trivial and blow-up solutions under the assumptions of Theorem~\ref{Th2}. Indeed,
let the conditions of Theorem~\ref{Th2} hold,
 $\beta (t) \equiv 0, \,$ $f(s) $ be a nondecreasing H\"{o}lder continuous function
 on $[0, \epsilon]$ for some $\epsilon >0$ and
 $$
 \int_0^{\epsilon}  \frac{ds}{f(s)} < + \infty.
 $$
As it was proved in \cite{LRS}, problem (\ref{1e})--(\ref{1i}) has trivial and positive for $t >0$
solutions and last one blows up in finite time by Theorem~\ref{Th2}.
 \end{remark}



\section{Global existence }\label{GE}

To formulate global existence result for problem (\ref{1e})--(\ref{1i}) we suppose:
\begin{equation}\label{1.12}
f(s) \,\, \textrm{is a nonnegative locally H\"{o}lder continuous function for} \,\,   s \geq  0,
\end{equation}
\begin{equation}\label{1.12b}
\textrm{there exists} \,\, p > 0 \,\, \textrm{such that} \,\, f(s) \,\, \textrm{is a positive nondecreasing function for} \,\, s \in (0, p),
\end{equation}
\begin{equation}\label{1.12a}
\int_0  \frac{ds}{f(s)} = + \infty,  \,\,\,  \lim_{s \to 0} \frac{g(s)}{s} = 0,
\end{equation}
\begin{equation}\label{1.13}
\int_0^{+\infty} \left(  \alpha (t)  + \beta (t) \right) \,dt < +\infty
\end{equation}
and there exist positive constants $\gamma,\;t_0$ and $K$ such that $\gamma>t_0$ and
\begin{equation}\label{1.14}
    \int_{t-t_0}^t {\frac{\beta (\tau) d\tau}{\sqrt{t-\tau}}} \leq  K \, \textrm{ for } \, t \geq\gamma.
\end{equation}

\begin{theorem}\label{Th3}
Let  (\ref{1.12})--(\ref{1.14}) hold. Then problem (\ref{1e})--(\ref{1i})
has bounded global solutions for small initial data.
\end{theorem}
\begin{proof}
It is well known that problem (\ref{1e}) -- (\ref{1i}) has a local nonnegative classical solution $u(x,t).$
Let $y(x,t)$ be a solution of the following problem
\begin{equation}\label{vsp}
\left\{
  \begin{array}{ll}
    y_t = \Delta y, \; x\in\Omega, \; t>0, \\
    \frac{\partial y(x,t)}{\partial \nu} = \xi (t) + \beta (t), \; x \in\partial\Omega, \; t>0, \\
    y(x,0)= 1,\; x\in\Omega,
  \end{array}
\right.
\end{equation}
where $\xi (t)$ is a positive continuous function that satisfies (\ref{1.13}), (\ref{1.14}) with $\beta (t) = \xi (t).$
According to Lemma 3.3 of \cite{GladkovKavitova} there exists
a positive constant $Y$ such that
\begin{equation*}\label{}
1 \leq y(x,t) \leq Y, \,  x\in\Omega, \; t>0.
\end{equation*}
Due to (\ref{1.12b}), (\ref{1.12a}) for any $a \in (0, p),$ there exist $\varepsilon (a)$ and a positive continuous function $\eta (t)$ such that
\begin{equation*}\label{}
0 < \varepsilon (a) < \frac{a}{Y},   \,\,\, \int_0^\infty  \eta (t) \, dt < \infty  \,\,\, \textrm{and}
\,\,\, \int_{\varepsilon Y}^{a}  \frac{ds}{f(s)} > Y \int_0^{\infty} \left(  \alpha (t) + \eta (t) \right)   \, dt
\end{equation*}
for any $\varepsilon \in (0, \varepsilon (a)).$
Now for any $T >0$ we construct a positive supersolution of
(\ref{1e})--(\ref{1i}) in $Q_T$ in such a form that
\begin{equation*}\label{}
\overline u (x,t) = \varepsilon z(t) y(x,t),
\end{equation*}
where function $z(t)$  is defined in the following way
\begin{equation*}\label{}
\int_{\varepsilon Y}^{\varepsilon Y z(t)}  \frac{ds}{f(s)} = Y \int_0^t    \left(  \alpha (\tau) + \eta (\tau) \right) \, d\tau.
\end{equation*}
It is easy to see that $\varepsilon Y z(t) < a$ and  $z(t)$ is the solution of the following Cauchy problem
\begin{equation*}\label{}
z' (t) - \frac{1}{\varepsilon} \left(  \alpha (t) + \eta (t) \right)  f(\varepsilon Y z(t)) = 0, \,\,\,  z(0) = 1.
\end{equation*}
After simple computations it follows that
\begin{eqnarray*}
        \overline u_t - \Delta \overline u - \alpha (t)  f(\overline u)
& = &  \varepsilon z' y + \varepsilon z y_t - \varepsilon z \Delta y - \alpha (t) f(\varepsilon  z y)\\
        &\geq& \alpha (t) (f(\varepsilon Y z(t)) - f(\varepsilon  z y))  + \eta (t) f(\varepsilon Y z(t)) > 0, \;x\in\Omega, \; t>0,
    \end{eqnarray*}
and
    \begin{eqnarray*}
\frac{\partial \overline u(x,t)}{\partial\nu} - \beta (t)  g(\overline u) & = &
 \varepsilon z(t) (\xi (t) +  \beta (t) ) - \beta (t)  g(\varepsilon z(t) y(x,t)) \\
 &>&  \varepsilon z(t) \beta (t) \left[ 1- \frac{g(\varepsilon z(t) y(x,t))}{\varepsilon z(t) y(x,t)} y(x,t) \right] \geq 0
    \end{eqnarray*}
for small values of $a.$ Thus, by Theorem~\ref{comp-prins} there
exist bounded global solutions of (\ref{1e})--(\ref{1i}) for any initial data
satisfying the inequality
\begin{equation*}\label{}
u_0 (x) < \varepsilon.
\end{equation*}
\end{proof}

\begin{remark}
We suppose that $g(s)$ is a nondecreasing positive function for $s >0,$
$f(s) >0$ for $s >0$ and (\ref{01}), (\ref{21}) hold.
Then by Theorem~\ref{Th1} and Theorem~\ref{Th2}
(\ref{1.13}) is necessary for global existence of
solutions of (\ref{1e})--(\ref{1i}).

Let for any $a > 0$ $g(s) > \delta (a) > 0$  if $s > a.$
Then arguing in the same way as in the proof of
Lemma~3.3 of \cite{GladkovKavitova} it is easy to show
that (\ref{1.14}) is necessary for the existence of nontrivial
bounded global solutions of (\ref{1e})--(\ref{1i}).

 \end{remark}

\end{document}